\documentclass{birkjour}
\usepackage{amsmath}
\usepackage{amsfonts,amssymb,mathrsfs}

 \newtheorem{thm}{Theorem}[section]

 \theoremstyle{definition}
 \newtheorem{defi}[thm]{Definition}
 \theoremstyle{remark}
 \newtheorem{rem}[thm]{Remark}
 
 \numberwithin{equation}{section}

\usepackage{pstricks}
\usepackage[all]{xy}
\usepackage{tikz}
\usepgflibrary{arrows}

\renewcommand{\figurename}{Fig.}

\input amssym.def

\renewcommand{\theequation}{\thesection.\arabic{equation}}
\numberwithin{equation}{section}

\renewcommand{\thefigure}{\thesection.\arabic{figure}}
\numberwithin{figure}{section}

\newcommand{\Gra}{{\mathsf{Gra} }}
\newcommand{\dGra}{{\mathsf{dGra} }}
\newcommand{\KGra}{{\mathsf {KGra} }}

\newcommand{\GC}{{\mathsf{GC}}}

\newcommand{\fGC}{{\mathsf{fGC}}}

\newcommand{\dfGC}{{\mathsf{dfGC}}}

\newcommand{\dgra}{{\rm d g r a }}

\newcommand{\OC}{{\mathsf{OC}}}

\renewcommand{\sc}{{\mathfrak{s} \mathfrak{c}}}

\newcommand{\GRT}{{\mathsf{GRT}}}

\newcommand{\Sh}{{\rm S h }}

\newcommand{\Der}{{\rm D e r}}

\newcommand{\wt}[1]{{\widetilde{#1}}}

\newcommand{\sCbu}{C^{\bullet+1}}
\newcommand{\Cbu}{C^{\bullet}}

\newcommand{\al}{{\alpha}}

\newcommand{\bul}{{\bullet}}

\newcommand{\ed}{{\bullet\hspace{-0.05cm}-\hspace{-0.05cm}\bullet}}

\newcommand{\mgl}{{\mathfrak{gl}}}

\newcommand{\si}{{\sigma}}
\newcommand{\ga}{{\gamma}}

\newcommand{\ve}{{\varepsilon}}

\newcommand{\G}{{\Gamma}}

\newcommand{\pa}{{\partial}}

\newcommand{\bs}{{\mathbf{s}}}

\newcommand{\cD}{{\mathcal {D}}}
\newcommand{\cA}{{\mathcal{A}}}

\newcommand{\cH}{{\mathcal{H}}}

\newcommand{\cO}{{\mathcal {O}}}
\newcommand{\cV}{{\mathcal {V}}}

\newcommand{\bbK}{{\mathbb K}}

\newcommand{\bbR}{{\mathbb R}}
\newcommand{\bbZ}{{\mathbb Z}}
\newcommand{\bbQ}{{\mathbb Q}}

\newcommand{\te}{\theta}
\newcommand{\Te}{\Theta}

\renewcommand{\L}{\langle}

\title[Exhausting quantization procedures]
 {Exhausting formal quantization procedures}

\author[V. A. Dolgushev]{Vasily Dolgushev}

\address{%
Department of Mathematics\\
Temple University\\
1805 N. Broad. St. Rm. 638\\
Philadelphia PA, 19130, USA}

\email{vald@temple.edu}



\date{}

\begin{document}

\maketitle

\begin{abstract}
In paper \cite{stable1} the author introduced stable formality quasi-iso\-mor\-phisms 
and described the set of its homotopy classes. This result can be interpreted as 
a complete description of formal quantization procedures. 
In this note we give a brief exposition of stable formality 
quasi-iso\-mor\-phisms and prove that every homotopy class 
of stable formality quasi-iso\-mor\-phisms contains a representative 
which admits globalization. 
This note is loosely based on the talk given 
by the author at XXX Workshop on Geometric Methods 
in Physics in Bialowieza, Poland.  
\end{abstract}

\section{Introduction}
In seminal paper \cite{K} M. Kontsevich constructed 
an $L_{\infty}$ quasi-iso\-mor\-phism from the graded Lie algebra of 
polyvector fields on the affine space $\bbR^d$ to the dg Lie algebra 
of Hochschild cochains $\Cbu(A)$ for the polynomial 
algebra $A = \bbR[x^1,x^2, \dots, x^d]$\,. This result implies 
that equivalence classes of star-products on $\bbR^d$ are 
in bijection with the equivalence classes of formal Poisson 
structures on $\bbR^d$\,. This theorem also implies that 
Hochschild cohomology of a deformation quantization 
algebra is isomorphic to the Poisson cohomology of 
the corresponding formal Poisson structure. 

In the view of these consequences, we will think about
$L_{\infty}$ quasi-iso\-mor\-phisms from the graded Lie algebra of 
polyvector fields on the affine space $\bbR^d$ to the dg Lie algebra 
of Hochschild cochains $\Cbu(A)$ as {\it formal quantization 
procedures}. 

Following \cite{erratum} one can define a natural notion 
of homotopy equivalence on the set of $L_{\infty}$-morphisms 
between dg Lie algebras (or even $L_{\infty}$-algebras). 
Furthermore, according to Lemma B.5 from \cite{BDW}, 
homotopy equivalent $L_{\infty}$ quasi-morphisms for 
$\Cbu(A)$ give the same bijection between the set of
equivalence classes of star-products and the set of 
equivalence classes of formal Poisson structures. 
Thus, for the purposes of applications, we should only 
be interested in homotopy classes of formality quasi-iso\-mor\-phisms. 

In paper \cite{stable1} the author developed a framework 
of what he calls {\it stable formality quasi-iso\-mor\-phisms  (SFQ)} 
and showed that homotopy classes of such SFQ's form 
a torsor for the group which 
is obtained by exponentiating the Lie algebra $H^0 (\GC)$ 
where $\GC$ is the graph complex introduced by 
M. Kontsevich in \cite[Section 5]{K-conj}. Any SFQ gives us 
an $L_{\infty}$ quasi-iso\-mor\-phism 
for the Hochschild cochains of $A= \bbR[x^1,x^2, \dots, x^d]$ 
in all\footnote{In fact they are also defined 
for any $\bbZ$-graded affine space.} dimensions $d$ simultaneously. 
Moreover, homotopy equivalent SFQ's
give homotopy equivalent  $L_{\infty}$ quasi-iso\-mor\-phisms 
for the Hochschild cochains of $A= \bbR[x^1,x^2, \dots, x^d]$\,. 
Thus the main result (Theorem 6.2) of \cite{stable1} can be 
interpreted as a complete description of formal quantization 
procedures in the stable setting. 

In the next section we remind the full (directed) graph complex
and its relation to Kontsevich's graph complex  $\GC$ 
\cite[Section 5]{K-conj}. In Section 3 we give a brief exposition of 
stable formality quasi-iso\-mor\-phisms (SFQ). Finally, in Section 4 we 
prove that every SFQ is homotopy equivalent to an SFQ which 
admits globalization.

~\\[-0.3cm]
{\bf Notation and conventions.} In this note we assume that 
the ground field $\bbK$ contains the field of reals.  
For most of algebraic structures considered 
in this note, the underlying symmetric monoidal category is 
the category of unbounded cochain complexes of $\bbK$-vector spaces.
For a cochain complex $\cV$ we denote 
by $\bs \cV$ (resp. by $\bs^{-1} \cV$) the suspension (resp. the 
desuspension) of $\cV$\,. In other words, 
$$
\big(\bs \cV\big)^{\bul} = \cV^{\bul-1}\,,  \qquad
\big(\bs^{-1} \cV\big)^{\bul} = \cV^{\bul+1}\,. 
$$
$\Cbu(A)$  denotes the Hochschild cochain
complex of an associative algebra  (or more generally an $A_{\infty}$-algebra) $A$
with coefficients in $A$\,. For a commutative ring $R$ and an $R$-module $V$ we denote 
by $S_R(V)$ the symmetric algebra of $V$ over $R$\,.

Given an operad $\cO$, we denote by $\circ_i$ the elementary 
operadic insertions: 
$$
\circ_i : \cO(n) \otimes \cO(k) \to \cO(n+k-1)\,, \qquad 1 \le i \le n\,.
$$

The notation $\Sh_{p,q}$ is reserved for the set of $(p,q)$-shuffles in $S_{p+q}$\,.
A graph is {\it directed} if each edge 
carries a chosen direction. A graph $\G$ with 
$n$ vertices is called {\it labeled} if $\G$ is equipped with a bijection between 
the set of its vertices and the set $\{1,2, \dots, n\}$\,.  
$\ve$ denotes a formal deformation parameter.

~\\[-0.4cm]
{\bf Acknowledgment.}
I would like to thank  Chris Rogers and  Thomas Willwacher
for  numerous illuminating discussions. 
This paper is loosely based on the talk given 
by the author at XXX Workshop on Geometric Methods 
in Physics in Bialowieza, Poland.   
The cost of my trip to this conference was partially 
covered by the NSF grant \# 1124929, which I acknowledge.
I am partially supported by the NSF grant DMS 0856196, 
and the grant FASI RF 14.740.11.0347.

\section{The full directed graph complex $\dfGC$} 
\label{sec:dfGC}
In this section we recall from \cite{Thomas} an extended version $\dfGC$ of 
Kontsevich's graph complex $\GC$ \cite[Section 5]{K-conj}.
For this purpose, we first introduce a collection of 
auxiliary sets $\{\dgra(n) \}_{n \ge 1}$\,.
An element of $\dgra_{n}$ is a directed labelled graph $\G$
with $n$ vertices and with the additional piece 
of data: the set of edges of $\G$ is equipped with a 
total order. An example of an element in $\dgra_4$ is 
shown on figure \ref{fig:exam}. 
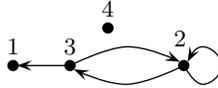
\begin{figure}[htp] 
\centering 
\psset{unit=0.5cm}
\begin{pspicture}(4,3)
\pscircle*(1,2){0.15}
\rput(1,2.5){\small $4$}
\pscircle*(-1.5,1){0.15}
\rput(-1.5,1.5){\small $1$}
\pscircle*(0,1){0.15}
\rput(0,1.5){\small $3$}
\pscircle*(3,1){0.15}
\rput(2.9,1.7){\small $2$}
\psline[arrowsize=4pt, linewidth=0.5pt]{->}(0,1)(-1.4,1)
\pscurve[arrowsize=4pt, linewidth=0.5pt]{->}(0,1)(1.5,1.5)(2.9,1.1)
\pscurve[arrowsize=4pt, linewidth=0.5pt]{<-}(0.1,0.9)(1.5,0.5)(3,1)
\pscurve[arrowsize=4pt, linewidth=0.5pt]{->}(3,1)(3.5, 0.5) (4,1) (3.5, 1.5)(3.1,1.1)
\end{pspicture} 
\caption{The edges are equipped with the order 
$(3,1) < (3,2) < (2,3) < (2,2)$ } \label{fig:exam}
\end{figure}

Next, we introduce a collection of graded vector spaces 
$\{\dGra(n) \}_{n \ge 1}$\,.
The space  $\dGra(n)$ is spanned by elements of 
$\dgra_n$, modulo the relation $\G^{\si} = (-1)^{|\si|} \G$
where the graphs $\G^{\si}$ and $\G$ correspond to the same
directed labelled graph but differ only by permutation $\si$
of edges. We also declare that 
the degree of a graph $\G$ in $\dGra(n)$ equals 
$-e(\G)$, where $e(\G)$ is the number of edges in $\G$\,.
For example, the graph $\G$ on figure \ref{fig:exam} has 
$4$ edges. Thus its degree is $-4$\,. 

Following \cite{Thomas}, the collection $\{\dGra(n) \}_{n \ge 1}$
forms an operad.
The symmetric group $S_n$ acts on $\dGra(n)$ in the 
obvious way by rearranging labels and the operadic 
multiplications are defined in terms of natural operations 
of erasing vertices and attaching edges to vertices.  

The operad $\dGra$ can be upgraded to a 2-colored operad 
$\KGra$ whose spaces\footnote{For more details, we refer the 
reader to \cite[Section 3]{stable1}.} 
are formal linear combinations
of graphs used by M. Kontsevich in \cite{K}.

We define the graded vector space $\dfGC$
by setting
\begin{equation}
\label{dfGC}
\dfGC = \prod_{n \ge 1}\bs^{2n-2} \Big( \dGra(n) \Big)^{S_n}\,.
\end{equation}

Next, we observe that the formula
\begin{equation}
\label{bullet}
\G \bullet \wt{\G} =
\sum_{\si \in \Sh_{k, n-1}} \si \big( \G \circ_1 \wt{\G} \big)
\end{equation}
$$
\G \in \Big( \dGra(n) \Big)^{S_n}\,, \qquad \qquad \wt{\G} \in \Big( \dGra(k) \Big)^{S_k}
$$
defines a degree zero $\bbK$-bilinear operation on 
$
\displaystyle \bigoplus_{n \ge 1}\bs^{2n-2} \Big( \dGra(n) \Big)^{S_n}
$
which extends in the obvious way to the graded vector space
$\dfGC$ \eqref{dfGC}.  

It is not hard to show that the operation \eqref{bullet}
satisfies axioms of the pre-Lie algebra and hence $\dfGC$
is naturally a Lie algebra with the bracket give by the formula
\begin{equation}
\label{bracket}
[\ga, \wt{\ga}] = \ga \bullet  \wt{\ga} - (-1)^{|\ga| |\wt{\ga}|}\, \wt{\ga} \bullet \ga \,,   
\end{equation}
where $\ga$ and $\wt{\ga}$ are homogeneous vectors in $\dfGC$\,.

A direct computation shows that the degree $1$
vector  
\begin{equation}
\label{G-ed}
\begin{tikzpicture}[scale=0.5, >=stealth']
\tikzstyle{w}=[circle, draw, minimum size=4, inner sep=1]
\tikzstyle{b}=[circle, draw, fill, minimum size=4, inner sep=1]
\draw (-2,0) node[anchor=center] {{$ \G_{\ed} = $}};
\node [b] (b1) at (0,0) {};
\draw (0,0.6) node[anchor=center] {{\small $1$}};
\node [b] (b2) at (1.5,0) {};
\draw (1.5,0.6) node[anchor=center] {{\small $2$}};
\draw [->] (b1) edge (b2);
\draw (3,0) node[anchor=center] {{$+$}};
\node [b] (bb1) at (4.5,0) {};
\draw (4.5,0.6) node[anchor=center] {{\small $2$}};
\node [b] (bb2) at (6,0) {};
\draw (6,0.6) node[anchor=center] {{\small $1$}};
\draw [->] (bb1) edge (bb2);
\end{tikzpicture}
\end{equation}
satisfies the Maurer-Cartan equation
$[\,\G_{\ed}, \G_{\ed}\,] = 0\,.$

Thus, $\dfGC$ forms a dg Lie algebra with the bracket \eqref{bracket}
and the differential 
\begin{equation}
\label{pa}
\pa = [\,\G_{\ed} , ~ ]\,.
\end{equation}

\begin{defi}
The cochain complex $(\dfGC, \pa)$ is called the 
full directed graph complex.   
\end{defi}

Let us observe that every undirected labeled graph
$\G$ with $n$ vertices and with a chosen order on the set of its edges can 
be interpreted as the sum of all directed labeled graphs $\G_{\al}$ in $\dgra(n)$
from which the graph $\G$ is obtained by forgetting 
directions on edges.  For example, 
\begin{equation}
\label{G-ed-new}
\begin{tikzpicture}[scale=0.5, >=stealth']
\tikzstyle{w}=[circle, draw, minimum size=4, inner sep=1]
\tikzstyle{b}=[circle, draw, fill, minimum size=4, inner sep=1]
\draw (-2,0) node[anchor=center] {{$ \G_{\ed} = $}};
\node [b] (b1) at (0,0) {};
\draw (0,0.6) node[anchor=center] {{\small $1$}};
\node [b] (b2) at (1.5,0) {};
\draw (1.5,0.6) node[anchor=center] {{\small $2$}};
\draw (b1) edge (b2);
\end{tikzpicture}
\end{equation} 

Thus, using undirected labeled graphs we may form 
a suboperad $\Gra$ inside $\dGra$ and the sub- dg Lie 
algebra 
\begin{equation}
\label{fGC}
\fGC = \prod_{n \ge 1}\bs^{2n-2} \Big( \Gra(n) \Big)^{S_n} \quad \subset \quad \dfGC 
\end{equation} 

\begin{defi}[M. Kontsevich, \cite{K-conj}]
\emph{Kontsevich's graph complex} $\GC$ is the subcomplex 
\begin{equation}
\label{GC}
\GC \subset \fGC
\end{equation}
formed by (possibly infinite) linear combinations of 
connected graphs $\G$ satisfying these two properties: 
{\it each vertex of $\G$ has valency $\ge 3$, and 
the complement to any vertex is connected.} 
\end{defi}
It is easy to see that $\GC$ is a sub- dg Lie algebra of $\fGC$\,.
Furthermore, following\footnote{See lecture notes \cite{notes} for more detailed 
exposition.} \cite{Thomas} we have 
\begin{thm}[T. Willwacher, \cite{Thomas}] 
\label{thm:H-dfGC}
The cohomology of $\dfGC$ can 
be expressed in terms of cohomology of $\GC$\,. More precisely,
\begin{equation}
\label{H-dfGC-GC}
H^{\bul} (\dfGC) = \bs^{-2}\,S\big(\bs^2\, \cH \big)
\end{equation}
where 
$$
\cH = H^{\bul}(\GC) ~ \oplus ~ \bigoplus_{m \ge 0} \bs^{4m-1} \bbK\,.
$$
\end{thm}

Using  decomposition \eqref{H-dfGC-GC}, it is not hard to see that 
\begin{equation}
\label{H-0}
H^0(\dfGC) \cong H^0(\GC)
\end{equation}
and the Lie algebra $H^0(\dfGC)$ is pro-nilpotent.

\section{Stable formality quasi-iso\-mor\-phisms}
Let $A = \bbK[x^1, x^2, \dots, x^d]$ be the algebra of functions
on the affine space $\bbK^d$ and let $V^{\bul}_A$ be the algebra of 
polyvector fields on $\bbK^d$
\begin{equation}
\label{V-A}
V^{\bul}_A = S_A \big( \bs\, \Der (A) \big)\,.
\end{equation}
Recall that $V^{\bul}_A$ is a free commutative algebra 
$V^{\bul}_A = \bbK[x^1, x^2, \dots, x^d, \te_1, \te_2, \dots, \te_d] $
over $\bbK$ in $d$ generators $x^1, x^2, \dots, x^d$ of 
degree zero and $d$ generators $\te_1, \te_2, \dots, \te_d$
of degree one.

It is know that $V^{\bul+1}_A$ is a graded Lie algebra.
The Lie bracket on $V^{\bul+1}_A$ is given by the formula: 
\begin{equation}
\label{Schouten}
[v, w]_S = (-1)^{|v|} \sum_{i=1}^d \frac{\pa v}{\pa \te_i} \frac{\pa w}{\pa x^i} 
-  (-1)^{|v||w| + |w|} \sum_{i=1}^d \frac{\pa w}{\pa \te_i} \frac{\pa v}{\pa x^i}\,.
\end{equation}
It is called the {\it Schouten bracket}. 

In plain English an $L_{\infty}$-morphism 
$U$ from $V^{\bul+1}_A$ to $\sCbu(A)$ is an 
infinite collection of maps 
\begin{equation}
\label{U-n}
U_n : \big( V^{\bul+1}_A \big)^{\otimes\, n} \to \sCbu(A)\,, \qquad n \ge 1
\end{equation}
compatible with the action of symmetric groups and
satisfying an intricate sequence of quadratic relations. The first relation 
says that $U_1$ is a map of cochain complexes, the second relation
says that $U_1$ is compatible with the Lie brackets up to 
homotopy with $U_2$ serving as a chain homotopy and so on.

Kontsevich's construction of such a sequence \eqref{U-n} is 
``natural'' in the following sense:  given polyvector fields
$v_1, v_2, \dots, v_n \in V^{\bul+1}_A$, 
the value 
\begin{equation}
\label{value}
U_n \big(v_1, v_2, \dots, v_n\big)(a_1, a_2, \dots, a_k)  
\end{equation}
of the cochain $U_n (v_1, v_2, \dots, v_n)$ on 
 polynomials  $a_1, a_2, \dots, a_k \in A$
is obtained via contracting all
indices of  derivatives of various orders of $v_1, \dots, v_n, a_1, \dots, a_k$
in such a way that the resulting map 
$$
(V^{\bul}_A)^{\otimes\, n} \otimes A^{\otimes\, k} \to A
$$
is $\mgl_d(\bbK)$-equivariant.
Thus each term in $U_n$ can be encoded by a directed graph with 
two types of vertices: vertices of one type are reserved for polyvector fields 
and vertices of another type are reserved for polynomials. 
 
Motivated by this observation, 
the author introduced  in \cite{stable1}  a notion of 
{\it stable formality quasi-iso\-mor\-phism (SFQ)} which 
formalizes  $L_{\infty}$ quasi-iso\-mor\-phisms 
$U$ for Hochschild cochains satisfying this property:
{\it  each term in 
$U_n$ is encoded by a graph with two types of vertices
and all the desired relations on $U_n$'s hold universally, i.e. on 
the level of linear combinations of graphs.}

The precise definition of SFQ
is given in terms of 2-colored dg operads $\OC$ and $\KGra$\,.
The later operad $\KGra$ is a 2-colored extension 
of the operad $\dGra$ which is ``assembled'' from graphs used 
by M. Kontsevich in \cite{K}. This operad comes with 
a natural action on the pair $(V^{\bul+1}_A, A= \bbK[x^1, \dots, x^d])$\,. 
The operad  $\OC$ governs open-closed homotopy 
algebras introduced in \cite{OCHA} by  H. Kajiura and J. Stasheff.
We recall that an open-closed homotopy 
algebra is a pair $(\cV, \cA)$ of 
cochain complexes equipped with the following data: 
\begin{itemize}

\item An $L_{\infty}$-structure on $\cV$;

\item an $A_{\infty}$-structure on $\cA$; and 

\item an $L_{\infty}$-morphism from $\cV$ to the 
Hochschild cochain complex $\Cbu(\cA)$ of the 
$A_{\infty}$-algebra $\cA$\,.

\end{itemize}

Since the operad $\KGra$ acts on the pair  
$(V^{\bul+1}_A, A= \bbK[x^1, \dots, x^d])$, 
any morphism of dg operads
\begin{equation}
\label{F}
F : \OC \to  \KGra
\end{equation}
gives us an  $L_{\infty}$-structure on $V^{\bul+1}_A$, 
an  $A_{\infty}$-structure on $A$ and an $L_{\infty}$ morphism 
from $V^{\bul+1}_A$ to $\Cbu(A)$\,.

An SFQ is defined as 
a morphism \eqref{F} of dg operads 
satisfying three boundary conditions. 
The first condition guarantees that the $L_{\infty}$-algebra 
structure on $V^{\bul+1}_A$ induced by $F$ coincides with the 
Lie algebra structure given by the Schouten bracket \eqref{Schouten}. 
The second condition implies that the $A_{\infty}$-algebra 
structure on $A$ coincides with the usual associative
(and commutative) algebra structure on polynomials. 
Finally, the third condition ensures that the $L_{\infty}$-morphism 
$$
U : V^{\bul+1}_A \leadsto \sCbu(A)
$$
induced by $F$ starts with the Hochschild-Kostant-Rosenberg 
embedding. In particular, the last condition implies that $U$
is an $L_{\infty}$ quasi-iso\-mor\-phism.

Kontsevich's construction \cite{K} provides us with an example 
of an SFQ over any extension of the field of 
reals\footnote{The existence of an SFQ over rationals is proved 
in papers \cite{overQ} and \cite{stable11}.} 
  
In paper \cite{stable1} the author also defined the notion of 
homotopy equivalence for SFQ's. This notion is motivated by 
the property that $L_{\infty}$ quasi-iso\-mor\-phisms 
$$
U,  \wt{U}: V^{\bul+1}_A \leadsto  \sCbu(A)
$$    
corresponding to homotopy equivalent SFQ's $F$ and $\wt{F}$ 
are connected by a homotopy which ``admits a graphical 
expansion'' in the above sense.

Following \cite{K-conj} we have a chain map $\Theta$ from the 
full (directed) graph complex $\dfGC$ to the deformation complex 
of the dg Lie algebra $V^{\bul+1}_A$ of polyvector fields. 
In particular, every degree zero cocycle in $\dfGC$ produces 
an $L_{\infty}$-derivation of  $V^{\bul+1}_A$\,.
Exponentiating these $L_{\infty}$-derivations we get 
an action of the (pro-uni\-po\-tent) group 
$$
\exp\, \Big( \dfGC^0 \cap \ker \pa \Big)
$$ 
on the set of $L_{\infty}$ quasi-iso\-mor\-phisms 
\begin{equation}
\label{U}
U : V^{\bul+1}_A  \leadsto \sCbu(A) 
\end{equation}
for $A  = \bbK[x^1, \dots, x^d]$\,. Namely, given a cocycle $\ga \in \dfGC^0$,
the action of $\exp(\ga)$ is defined by the formula
\begin{equation}
\label{action}
U \mapsto  U \circ \exp\big(-\Theta(\ga)\big)\,,
\end{equation}
where $\Theta$ is the chain map from $\dfGC$ to the deformation 
complex of $V^{\bul+1}_A$\,.

In \cite{stable1}, it was proved that the action \eqref{action} 
descends to an action of the (pro-unipotent) group 
\begin{equation}
\label{exp-H0}
\exp \big(H^0(\dfGC)\big)
\end{equation}
on the set of homotopy classes of SFQ's. Moreover, 
\begin{thm}[Theorem 6.2, \cite{stable1}]
\label{thm:torsor}
The group \eqref{exp-H0} acts simply transitively on the set of 
homotopy classes of SFQ's. 
\end{thm}
In the view of philosophy outlined in the Introduction,  
this result can be interpreted as a complete description 
of formal quantization procedures. 

\begin{rem}
According to the recent result \cite[Theorem 1]{Thomas} of T. Willwacher, 
$\exp \big( H^0 (\GC)\big)$ is isomorphic to the
Grothendieck-Teichmueller group 
$\GRT$ introduced by V. Drinfeld in \cite{Drinfeld}. Thus, combining 
this result with Theorem \ref{thm:torsor}, we conclude that 
formal quantization procedures are ``governed'' by the group $\GRT$.
\end{rem}

\begin{rem}
In recent preprint \cite{StableThomas} Thomas Willwacher computes 
stable cohomology of the graded Lie algebra of polyvector fields
with coefficients in the adjoint representation. His computations 
partially justify the name ``stable formality quasi-isomorphism'' 
chosen by the author in \cite{stable1}. In particular, Thomas 
Willwacher mentions in  \cite{StableThomas} a possibility to deduce the 
part about transitivity from Theorem \ref{thm:torsor} in a more conceptual way. 
\end{rem}

\section{Globalization of stable formality quasi-iso\-mor\-phisms}
Given an $L_{\infty}$ quasi-iso\-mor\-phism \eqref{U}
for $A  = \bbK[x^1, \dots, x^d]$ we can 
ask the question of whether we can use it to construct a sequence 
of $L_{\infty}$ quasi-iso\-mor\-phisms which connects 
the sheaf $V^{\bul+1}_X$ of polyvector fields to the 
sheaf $\cD^{\bul+1}_X$ of polydifferential operators on 
a smooth algebraic variety $X$ over $\bbK$\,. 
There are several similar constructions \cite{CEFT}, 
\cite{VdB}, \cite{Ye} which allow us to produce such a sequence
under the assumption that
the $L_{\infty}$ quasi-iso\-mor\-phism \eqref{U} satisfies the 
following properties:

\begin{enumerate}

\item[{\bf A)}] One can replace  $A  = \bbK[x^1, \dots, x^d]$ in \eqref{U} by 
its completion  $A_{formal} = \bbK[[x^1, \dots, x^d]]$ ;

\item[{\bf B)}] the structure maps $U_n$ of $U$ are 
$\mgl_d(\bbK)$-equivariant;

\item[{\bf C)}] if $n >1$ then
\begin{equation}
\label{U-n-vect}
U_n(v_1, v_2, \dots, v_n) =  0
\end{equation}
for every set of vector fields $v_1, v_2, \dots, v_n \in \Der(A_{formal})$;

\item[{\bf D)}] if $n \ge 2$ and $v \in \Der(A_{formal})$ has the form 
$$
v = \sum_{i,j =1}^{d} v^i_j x^j \frac{\pa}{\pa x^i}\,, \qquad v^i_j \in \bbK 
$$
then for every set  $w_2, \dots, w_n \in V^{\bul+1}_{A_{formal}}$
\begin{equation}
\label{v-linear}
U_n(v, w_2, \dots, w_n) =  0\,.
\end{equation}

\end{enumerate}

In paper \cite{Chern} it was shown that for every degree zero 
cocycle $\ga \in \GC$ the structure maps $\Te(\ga)_n$ of the 
$L_{\infty}$-derivation $\Te(\ga)$ satisfy these properties:

\begin{enumerate}

\item[{\bf a)}]  $\Te(\ga)$ can be viewed as an $L_{\infty}$-derivation
of $V^{\bul+1}_{A_{formal}}$ with $A_{formal} = \bbK[[x^1, \dots, x^d]]$ ;

\item[{\bf b)}] the structure maps $\Te(\ga)_n$ of $\Te(\ga)$ are 
$\mgl_d(\bbK)$-equivariant;

\item[{\bf c)}] if $n >1$ then
\begin{equation}
\label{Tega-n-vect}
\Te(\ga)_n(v_1, v_2, \dots, v_n) =  0
\end{equation}
for every set of vector fields $v_1, v_2, \dots, v_n \in \Der(A_{formal})$;

\item[{\bf d)}] if $n \ge 2$ and $v \in \Der(A_{formal})$ has the form 
$$
v =  \sum_{i,j =1}^{d}  v^i_j x^j \frac{\pa}{\pa x^i}\,, \qquad v^i_j \in \bbK 
$$
then for every set  $w_2, \dots, w_n \in V^{\bul+1}_{A_{formal}}$
\begin{equation}
\label{v-linear-Tega}
\Te(\ga)_n (v, w_2, \dots, w_n) =  0\,.
\end{equation}

\end{enumerate}
Properties {\bf a)} and {\bf b)} are obvious, while properties {\bf c)}
and {\bf d)} follow from the fact that each graph in the linear combination 
$\ga\in \GC$ has only vertices of valencies $\ge 3$\,. 

Using these properties of $\Te(\ga)$ together with Theorems 
\ref{thm:H-dfGC} and \ref{thm:torsor} we deduce the main 
result of this note: 
\begin{thm}
\label{thm:main}
Every homotopy class of SFQ's 
contains a representative which can be used to 
construct a sequence of $L_{\infty}$ quasi-iso\-mor\-phisms 
connecting the sheaf  $V^{\bul+1}_X$ of polyvector fields to the 
sheaf $\cD^{\bul+1}_X$ of polydifferential operators on 
a smooth algebraic variety $X$ over $\bbK$\,.
\end{thm}
\begin{proof}
Let $F'$ be an SFQ. Our goal is to prove that 
the homotopy class of $F'$ contains a representative 
$F$ whose corresponding $L_{\infty}$ quasi-iso\-mor\-phism  
\eqref{U} satisfies  Properties {\bf A)} -- {\bf D)}  listed above. 

Let us denote by $F_K$ an SFQ whose corresponding 
$L_{\infty}$ quasi-iso\-mor\-phism
\begin{equation}
\label{U-K}
U_K : V^{\bul+1}_A  \leadsto \sCbu(A) 
\end{equation}
satisfies Properties {\bf A)} -- {\bf D)}. (For example, we can 
choose the SFQ coming from Kontsevich's construction \cite{K}.)

Theorem \ref{thm:torsor} implies that there exists a 
degree zero cocycle $\ga' \in \dfGC$ such that 
$F'$ is homotopy equivalent to the SFQ
\begin{equation}
\label{gapr-F-K}
\exp(\ga') \big(F_K \big)\,.
\end{equation}

On the other hand, we have isomorphism \eqref{H-0}. 
Therefore, $\ga'$ is cohomologous to a cocycle 
$\ga \in \GC$ and hence $F'$ is homotopy equivalent to 
\begin{equation}
\label{ga-F-K}
\exp(\ga) \big(F_K \big)\,.
\end{equation}

Since the $L_{\infty}$-derivation $\Te(\ga)$ satisfies 
Properties {\bf a)} -- {\bf d)} and the $L_{\infty}$ quasi-iso\-mor\-phism 
\eqref{U-K} satisfies Properties  {\bf A)} -- {\bf D)}, we conclude that
the $L_{\infty}$ quasi-iso\-mor\-phism corresponding to the SFQ \eqref{ga-F-K}
also satisfies Properties  {\bf A)} -- {\bf D)}.
 
Theorem \ref{thm:main} is proved. 
\end{proof}

\end{document}